\documentclass[12pt]{amsart}

\usepackage{graphicx}
\usepackage{amsmath}
\usepackage{amscd}
\usepackage{amsfonts}
\usepackage{amssymb}
\usepackage{color}

\numberwithin{equation}{section}

\newtheorem{theorem}{Theorem}[section]

\newtheorem{proposition}[theorem]{Proposition}

\newtheorem{definition}[theorem]{Definition}
\newtheorem{corollary}[theorem]{Corollary}

\theoremstyle{definition}

\newtheorem{remark}[theorem]{Remark}

\newcommand{\be}{\begin{equation}}
\newcommand{\ee}{\end{equation}}
\newcommand{\bes}{\begin{equation*}}
\newcommand{\ees}{\end{equation*}}

\newcommand{\cP}{\mathcal{P}}
\newcommand{\cQ}{\mathcal{Q}}

\newcommand{\lel}{\left\langle}
\newcommand{\rir}{\right\rangle}

\newcommand{\mb}[1]{\mathbb{#1}}

\renewcommand{\phi}{\varphi}

\begin{document}

\title{Unitary $N$-dilations for tuples of commuting matrices}

\author{John E. M\raise.45ex\hbox{\tiny{c}}Carthy}
\address{Dept. of Mathematics\\Washington University\\
St. Louis, MO 63130\\
USA}
\email{mccarthy@wustl.edu}

\author{Orr Moshe Shalit}

\address{Pure Mathematics Dept.\\
University of Waterloo\\
Waterloo, ON\; N2L--3G1\\
CANADA}
\email{oshalit@uwaterloo.ca}

\begin{abstract}
We show that whenever a contractive $k$-tuple $T$ on a finite dimensional space $H$
has a unitary dilation, then for any fixed degree $N$ there is a unitary $k$-tuple $U$ on a
finite dimensional space so that $q(T) = P_H q(U) |_H$ for all polynomials $q$ of degree at most $N$. 
\end{abstract}

\maketitle

\section{Introduction and statement of the main results}


Let $T_1, \ldots, T_k$ be a $k$-tuple of commuting contractions on a Hilbert space $H$. A $k$-tuple $U_1, \ldots, U_k$ of commuting unitaries on a Hilbert space $K$ is said to be a \emph{unitary dilation for} $T_1, \ldots, T_k$ if $H$ is a subspace of $K$ and if for all $n_1, \ldots, n_k \in \mb{N}$ the operator $T_1^{n_k} \cdots T_k^{n_k}$ is the compression of $U_1^{n_1} \cdots U_k^{n_k}$ onto $H$, meaning that
\bes
T_1^{n_k} \cdots T_k^{n_k} = P_H U_1^{n_1} \cdots U_k^{n_k} \big|_H .
\ees
(Here and below we are using the notation $P_H$ for the orthogonal projection of $K$ onto $H$).  
The problem of determining if a $k$-tuple of contractions has a unitary dilation is well studied, and has had a profound impact on operator theory; see \cite{SzNF70,Paulsen,Arv69,Arv72,Arv10} for example. The basic results in the theory are that a unitary dilation always exists when $k=1$ (this is Sz.-Nagy's unitary dilation theorem \cite{SzNDil}) and when $k=2$ (this is And\^{o}'s dilation theorem \cite{Ando}); also, Parrott \cite{Parrott} gave an example showing that when $k=3$ a unitary dilation might not exist.

In case there is a dilation, it can be shown that if $T_i$ is not a unitary for some $i$ then $K$ has to be infinite dimensional. Now, one reason to seek a unitary dilation for a given $k$-tuple of commuting operators, is to better understand $T_1, \ldots, T_k$ as a ``piece" of the $k$-tuple $U_1, \ldots, U_k$ --- given that $k$-tuples of commuting unitaries are particularly well understood. On the other hand, in the case when $T_1, \ldots, T_k$ act on a finite dimensional space, it is not entirely clear that unitaries acting on an infinite dimensional space are really better understood. One is naturally let to consider a dilation theory that involves only finite dimensional Hilbert spaces.

\begin{definition}
Let $T_1,\ldots,T_k$ be commuting contractions on $H$, and let $N \in \mb{N}$. A \emph{unitary $N$-dilation for $T_1, \ldots, T_k$} is a $k$-tuple of commuting unitaries $U_1,\ldots,U_k$ acting on a space $K\supseteq H$ such that
\be\label{eq:multidil}
T_1^{n_1} \cdots T_k^{n_k} = P_H U_1^{n_1} \cdots U_k^{n_k} P_H ,
\ee
for all non-negative integers $n_1, \ldots, n_k$ satisfying $n_1 + \ldots + n_k \leq N$.
\end{definition}

We will only be interested in $N$-dilations acting on finite dimensional spaces.
In \cite{Egervary} Egerv\'ary showed that every contraction $T$ on a finite dimensional space has a unitary $N$-dilation acting on a finite dimensional space --- this should be thought of as a finite dimensional version of Sz.-Nagy's dilation theorem. In \cite{LevyShalit10} this idea was revisited and some consequences were explored. The goal of this paper is to show that an analogue of Egerv\'ary's result holds for two commuting operators --- this can be thought of as a finite dimensional version of And\^{o}'s dilation theorem. We obtain several other related results as well, as explained below.

\subsection{Dilations and $N$-dilations}
It will be convenient to denote by $\cP_N(k)$ the space of complex polynomials in $k$ variables of degree less than or equal to $N$. Put $\cP(k) = \bigcup_N \cP_N(k)$. 

Section \ref{sec2} will be devoted to the proof of the following theorem:

\begin{theorem}\label{thm:dilations}
Let $H$ be a Hilbert space, $\dim H = n$. Let $T_1,\ldots,T_k$ be commuting contractions on $H$. 
The following are equivalent:
\begin{enumerate}
\item The $k$-tuple $T_1,\ldots,T_k$ has a unitary dilation.
\item For every $N$, the $k$-tuple $T_1,\ldots,T_k$ has a unitary $N$-dilation that acts on a finite dimensional space.
\end{enumerate}
When the conditions hold, the regular unitary $N$-dilation can be taken to act on a Hilbert space of dimension $n^2(n+1)\frac{(N+k)!}{N!k!} + n$.
\end{theorem}

And\^{o}'s dilation theorem \cite{Ando} asserts that every pair of commuting contractions has a unitary dilation. Thus the above theorem immediately implies the following.

\begin{corollary}\label{thm:Ando}
Let $H$ be a Hilbert space, $\dim H = n$. Let $A,B$ be two commuting contractions on $H$. Then for all $N$, there is a Hilbert space $K$ containing $H$, with $\dim K = n^2(n+1)\frac{(N+1)(N+2)}{2} + n$, and two commuting unitaries $U,V$ on $K$ such that
\[q(A,B) = P_H q(U,V) \big|_H \]
for all $q \in \cP_N(k)$.
\end{corollary}

It is known that any $k$-tuple of commuting $2\times 2$ contractions has a unitary dilation (see \cite[p. 21]{Drury} or \cite{Hol90}). Thus we also obtain the next corollary.

\begin{corollary}\label{thm:2X2}
Let $H$ be a Hilbert space, $\dim H = 2$. Let $T_1,\ldots,T_k$ be commuting contractions on $H$. Then for all $N$, there is a Hilbert space $K$ containing $H$, with $\dim K = 12\frac{(N+k)!}{N!k!} + 2$, and a $k$-tuple of commuting unitaries $U_1, \ldots, U_k$ on $K$ such that
\[q(T_1, \ldots, T_k) = P_H q(U_1, \ldots, U_k) \big|_H \]
for all $q \in \cP_N(k)$.
\end{corollary}

\subsection{Regular dilations and regular $N$-dilations}
When does a commuting $k$-tuple of contractions have a unitary dilation? This question has received a lot of attention (see for example Chapter I of \cite{SzNF70}, \cite{Opela}, \cite{StoSza} or \cite{CiStoSza}), but it is fair to say that a completely satisfying answer has not yet been found. 
However, there is a stronger notion of dilation --- \emph{regular dilation} --- for which a simple necessary and sufficient condition is known. Before defining regular dilations we introduce some notation.

Let $\cQ^0_N(k)$ denote the set of all functions $f$ on $k$ variables $z_1, \ldots, z_k$ of the form 
\[f(z_1, \ldots, z_k) = q(y_1, \ldots, y_k) , \]
for some $q \in \cP_N(k)$, where for each $i$, $y_i$ is either $z_i$ or $\overline{z_i}$. That is $f$ is an analytic polynomial in some of the variables, and a co-analytic polynomial in the rest of them. Let $\cQ_N(k)$ denote the space spanned by $\cQ^0_N(k)$.  Put $\cQ(k) = \bigcup_N \cQ_N(k)$.

If $T_1, \ldots, T_k$ is a $k$-tuple of commuting contractions, then for any $f \in \cQ(k)$ we define $f(T_1, \ldots, T_k)$ as follows. If $f$ is a monomial of the form 
$$f(z_1, \ldots, z_k) = z_{i_1}^{n_1} \cdots z_{i_s}^{n_s} \overline{z_{j_1}^{m_1} \cdots z_{j_t}^{m_t}} $$ 
(where necessarily $\{i_1, \ldots, i_s\} \cap \{j_1, \ldots, j_t\} = \emptyset$) then we define
\[f(T_1, \ldots, T_k) =  (T_{j_1}^{m_1} \cdots T_{j_t}^{m_t})^* T_{i_1}^{n_1} \cdots T_{i_s}^{n_s} .\]
This definition is expanded linearly to all of $\cQ(k)$.

\begin{definition}
Let $T_1,\ldots,T_k$ be commuting contractions on $H$. A \emph{regular unitary dilation for $T_1, \ldots, T_k$} is a $k$-tuple of commuting unitaries $U_1,\ldots,U_k$ acting on a space $K\supseteq H$ such that
\[f(T_1, \ldots, T_k) = P_H f(U_1, \ldots, U_k) \big|_H \]
for all $f \in \cQ(k)$.
\end{definition}

It is known \cite[Theorem I.9.1]{SzNF70} that a necessary and sufficient condition for a $k$-tuple $T_1, \ldots, T_k$ to have a regular unitary dilation is that for all $S \subseteq \{1, \ldots, k\}$, we have the operator inequality
\be\label{eq:oi}
\sum_{I \subseteq S} (-1)^{|I|}T^*_{i_1} \cdots T_{i_m}^* T_{i_1} \cdots T_{i_m} \geq 0,
\ee
where $I = \{i_1, i_2, \ldots, i_m\}$.


\begin{definition}
Let $T_1,\ldots,T_k$ be commuting contractions on $H$, and let $N \in \mb{N}$. A \emph{regular unitary $N$-dilation for $T_1, \ldots, T_k$} is a $k$-tuple of commuting unitaries $U_1,\ldots,U_k$ acting on a space $K\supseteq H$ such that
\[f(T_1, \ldots, T_k) = P_H f(U_1, \ldots, U_k) \big|_H \]
for all $f \in \cQ_N(k)$.
\end{definition}

In \cite{LevyShalit10} it was proved that every $k$-tuple of {\em doubly commuting} contractions has a regular unitary $N$-dilation acting on a finite dimensional space, for all $N$, and the question was raised, whether condition (\ref{eq:oi}) is necessary and sufficient for a regular unitary $N$-dilation to exist for all $N$. In this paper we show that indeed it is, a result which follows from the following theorem (to be proved in Section \ref{sec:reg}).

\begin{theorem}\label{thm:reg_dilations}
Let $H$ be a Hilbert space, $\dim H = n$. Let $T_1,\ldots,T_k$ be commuting contractions on $H$. 
The following are equivalent:
\begin{enumerate}
\item The $k$-tuple $T_1,\ldots,T_k$ has a regular unitary dilation.
\item For every $N$, the $k$-tuple $T_1,\ldots,T_k$ has a regular unitary $N$-dilation which acts on a finite dimensional space.
\end{enumerate}
When the conditions hold, the regular unitary $N$-dilation can be taken to act on a Hilbert space of dimension $n^2(n+1) \times \dim \cQ_N(k)+ n$.
\end{theorem}

From the theorem and the preceding discussion we obtain:

\begin{corollary}
Let $H$ be a Hilbert space, $\dim H = n$. Let $T_1,\ldots,T_k$ be commuting contractions on $H$. 
The following are equivalent:
\begin{enumerate}
\item For every $N$, the $k$-tuple $T_1,\ldots,T_k$ has a regular unitary $N$-dilation which acts on a finite dimensional space.
\item Condition (\ref{eq:oi}) holds for all $S \subseteq \{1, \ldots, k\}$.
\end{enumerate}
\end{corollary}

\subsection{A formula for the dilation} The proofs for Theorems \ref{thm:dilations} and \ref{thm:reg_dilations} provided below are non-constructive, and provide little information on how to effectively construct the unitary dilations. This should be contrasted with the results of \cite{LevyShalit10}, all of which had proofs involving concrete, finite dimensional constructions. For example, the $N$-dilation which Egerv\'ary constructs for a contraction $T$ is given by 
\[
U = \begin{pmatrix}
T     &  &  &  & D_{T^*} \\
D_T &  &  &  & -T \\
     & I &   &  &  \\
     &  & \ddots  &  &  \\
     &  &   & I & 0 
\end{pmatrix},
\]
where $D_T = (I-T^*T)^{1/2}$ and $D_{T^*} = (I-TT^*)^{1/2}$, and the matrix $U$ has $(N+1) \times (N+1)$ blocks. It would be interesting to obtain such concrete formulas for the dilation of two commuting contractions. After all, we have a bound on the dimension of the space on which the dilation acts.

One of the goals behind this work was to find a proof of And\^{o}'s inequality for two commuting matrices that does not involve And\^{o}'s dilation theorem or infinite dimensional spaces. This goal has not yet been met.

\section{Proof of Theorem \ref{thm:dilations}}


Let the notation be as in Theorem \ref{thm:dilations}.
\label{sec2}
\subsection{Existence of $N$-dilations implies existence of dilations}\label{subsec:onedir}

Suppose that for all $N$, the $k$-tuple $T_1, \ldots, T_k$ has a unitary $N$-dilation. Fixing $N$, let $U_1, \ldots, U_k$ be a unitary $N$-dilation acting on a (finite dimensional) Hilbert space $K$. If $Q = (q_{i,j})_{i,j=1}^m$ is an $m \times m$ matrix with entries in $\cP_N(k)$, then we have 
\[q_{i,j} (T_1, \ldots, T_k) = P_H q_{i,j}(U_1, \ldots, U_k) \big |_H \]
for all $i,j$, therefore 
\[\Vert \left(q_{i,j} (T_1, \ldots, T_k) \right)_{i,j}\Vert \leq \Vert \left(q_{i,j} (U_1, \ldots, U_k) \right)_{i,j}\Vert ,\]
where the norm is the operator norm on $\underbrace{H \oplus \ldots \oplus H}_{m \textrm{ times}}$ (on the left hand side) or on $K \oplus \ldots \oplus K$ (on the right hand side). The right hand side is less than or equal to $\sup \{\| (q_{i,j}(z))_{i,j}\|_{M_m} : z \in \mb{T}^k \}$, therefore 
\[\Vert \left(q_{i,j} (T_1, \ldots, T_k) \right)_{i,j}\Vert \leq \sup \{\| (q_{i,j}(z))_{i,j}\|_{M_m} : z \in \mb{T}^k \} . \]
Since this holds for all $N$, we find that the polydisc $\overline{\mb{D}}^k$ is a {\em complete spectral set} for $T_1, \ldots, T_k$, so by Arveson's dilation theorem (see \cite[pp. 278--279]{Arv72} or \cite[pp. 86--87]{Paulsen})
the tuple $T_1, \ldots, T_k$ has a unitary dilation.

\subsection{Existence of dilations implies existence of $N$-dilations}

We begin by proving a proposition of independent interest.

\begin{proposition}\label{prop:point_eval}
Suppose that $T_1, \ldots, T_k$ are commuting contractions on $H$, $\dim H = n$, and that they have a unitary dilation $U_1, \ldots, U_k$. Fix $N \in \mb{N}$. Then there is an integer $M$, there are $M$ points $w_i, \ldots, w_M \in \mb{T}^k$ and there are $M$ non-negative operators $A_1, \ldots, A_M \in B(H)$ with $\sum A_i = I_H$, such that
\be\label{eq:point_eval}
q(T_1, \ldots, T_k) = \sum_{i=1}^M q(w_i) A_i \,\,, \,\, \textrm{ for all } q \in \cP_N(k).
\ee
The integer $M$ can be taken to be $n(n+1)\frac{(N+k)!}{N!k!} + 1$.
\end{proposition}

\begin{proof}
We begin by proving a relation of the type (\ref{eq:point_eval}) for the tuple $rT_1, \ldots, rT_k$, where $r \in (0,1)$. Note that the tuple $rU_1, \ldots, rU_k$ is a (normal) dilation for $rT_1, \ldots, rT_k$, that is
\be\label{eq:norm_dil}
q(rT_1, \ldots, rT_k) = P_H q(rU_1, \ldots, rU_k) P_H 
\ee
for all $q \in \cP(k)$.

Let $\Phi(z;w)$ be the Poisson kernel for the polydisc $\mb{D}^k$ (see, e.g., \cite[p. 17]{RudinPoly}); $\Phi : \mb{D}^k \times \mb{T}^k \rightarrow \mb{R}$ has the following properties:
\begin{enumerate}
\item $\Phi(z;w) > 0$ for $z \in \mb{D}^k, w \in \mb{T}^k$,
\item $\int_{\mb{T}^k} \Phi(z;w) dw = 1$ for all $z \in \mb{D}^k$, where $dw$ denotes normalized Lebesgue measure on $\mb{T}^k$,
\item If $u$ is $n$-harmonic (that is, harmonic in each complex variable separately) in a neighborhood of $\overline{\mb{D}}^k$ then for all $z \in \mb{D}^k$, 
\[ 
u(z) = \int_{\mb{T}^k} \Phi(z;w)u(w) dw .
\]
\end{enumerate}
In particular, the functional calculus for the normal $k$-tuple $rU_1, \ldots, rU_k$ then gives
\[
q(rU_1, \ldots, rU_k) = \int_{\mb{T}^k} \Phi((rU_1, \ldots, rU_k);w)q(w) dw
\]
for all $q \in \cP_N(k)$. Combining this with (\ref{eq:norm_dil}) we obtain
\be\label{eq:Poisson}
q(rT_1, \ldots, rT_k) =  \int_{\mb{T}^k} P_H\Phi((rU_1, \ldots, rU_k);w)P_Hq(w) dw .
\ee
Let $\{e_1, \ldots, e_n\}$ be an orthonormal basis for $H$. Define 
\[
f_{ij}(w) = \lel P_H\Phi((rU_1, \ldots, rU_k);w)P_H e_j, e_i \rir .
\]

Let $V$ be the real vector space spanned by $\{\Re f_{ij}q, \Im f_{ij}q : 1\leq i \leq j \leq n, q \in \cP_N(k)\}$ (note that $f_{ij} = \overline{f_{ji}}$ because $P_H\Phi((rU_1, \ldots, rU_k);w)P_H$ is self-adjoint). The dimension of $V$ is at most $2 \times n(n+1)/2 \times \frac{(N+k)!}{N!k!}$ (the last factor is $\dim \cP_N(k)$). The linear functional on $V$ given by
\bes
g \mapsto \int_{\mb{T}^k} g(w) dw 
\ees
is in the convex hull of point evaluations on $\mb{T}^k$. By Carath\'{e}odory's Theorem \cite[p. 453]{DavDon}, there are $M := n(n+1)\frac{(N+k)!}{N!k!} + 1$ points $w^{(r)}_1\ldots, w^{(r)}_M $ in 
$ \mb{T}^k$ and $M$ positive numbers $a^{(r)}_1, \ldots, a^{(r)}_M$ summing to $1$ such that 
\be\label{eq:convex}
 \int_{\mb{T}^k} g(w) dw = \sum_{i=1}^M a^{(r)}_i g(w^{(r)}_i) \,\, , \,\, \textrm{ for all } g \in V.
\ee
Put 
$$A^{(r)}_i = a^{(r)}_i P_H\Phi((rU_1, \ldots, rU_k);w^{(r)}_i)P_H .$$ 
Combining (\ref{eq:Poisson}) and (\ref{eq:convex}) we obtain 
\bes
q(rT_1, \ldots, rT_k) = \sum_{i=1}^M q(w^{(r)}_i) A^{(r)}_i \,\,, \,\, \textrm{ for all } q \in \cP_N(k).
\ees
Letting $q \equiv 1$, we get $\sum_i A_i^{(r)} = I_H$, so the $A_i^{(r)}$ are 
positive and uniformly bounded. Hence
we get (\ref{eq:point_eval}) by a compactness argument.
\end{proof}

As a corollary we obtain the following sharpening of von Neumann's inequality.

\begin{corollary}\label{cor:vN}
Suppose that $T_1, \ldots, T_k$ are commuting contractions on $H$, $\dim H = n$, that have a unitary dilation $U_1, \ldots, U_k$. Fix $N \in \mb{N}$. Then there is an integer $M$ which is not greater than $n(n+1)\frac{(N+k)!}{N!k!} + 1$, and there are $M$ points $w_i, \ldots, w_M \in \mb{T}^k$ such that
\bes
\|q(T_1, \ldots, T_k)\| \leq \max_{1\leq i \leq M} |q(w_i)|  \,\,, \,\, \textrm{ for all } q \in \cP_N(k).
\ees
\end{corollary}

\begin{remark}
By And\^{o}'s dilation theorem, every pair of commuting contractions $T_1, T_2$ has a unitary dilation, and the above results apply. In the case where neither $T_1$ and $T_2$
have eigenvalues of unit modulus, then one may replace the Poisson kernel in the above proof with the  spectral measure of a unitary dilation obtained in the proof of \cite[Theorem 3.1]{AM05}, and it follows that the points $w_1, \ldots, w_M$ all lie in the {\em distinguished variety} associated with $T_1, T_2$. This means they all lie in the intersection $V \cap {\mb T}^2$,
where $V$ is a one dimensional algebraic set that depends
on $T_1$ and $T_2$ but not on $N$. 
\end{remark}

We can now complete the proof of Theorem \ref{thm:dilations}. Let $T_1, \ldots, T_k$ be commuting contractions on $H$, $\dim H = n$, that have a unitary dilation. Let $w_1, \ldots, w_M$ and $A_1, \ldots, A_M$ be as in the conclusion of Proposition \ref{prop:point_eval}. The points $w_i$ all lie on the $k$-torus $\mb{T}^k$, so we write
\bes
w_i = (w_i^1, \ldots, w_i^k) \,\, , \,\, i=1, \ldots, M,
\ees
where $w_i^j \in \mb{T}$ for all $i$ and $j$. Since $\sum_i A_i = I_H$, the $M$-tuple of operators $A_1, \ldots, A_M$ can be thought of as a positive operator valued measure (POVM) on an $M$ point set. By Naimark's Theorem \cite[Theorem 4.6]{Paulsen}, this POVM can be dilated to a spectral measure on an $M$-point set. This just means that there are $M$ orthogonal projections $E_1, \ldots, E_M$ on a Hilbert space $K$ such that $\sum_i E_i = I_K$, for which
\be\label{eq:A_i}
A_i = P_H E_i P_H \quad , \quad i = 1, \ldots, M.
\ee
The familiar proof of Naimark's Theorem via Stinespring's Theorem (see Chapter 4 of \cite{Paulsen}) shows that $K$ can be chosen to be at most $M \times n$ dimensional. For $j = 1, \ldots, k$ we define
\bes
U_j = \sum_{i=1}^M w_i^j E_i.
\ees
Clearly, $U_1, \ldots, U_k$ are commuting unitaries on $K$, and by (\ref{eq:point_eval}) and (\ref{eq:A_i}) they constitute an $N$-dilation for $T_1, \ldots, T_k$. That completes the proof of Theorem \ref{thm:dilations}.

\subsection{Additional remark}

It is not clear what the appropriate analogue of ``commutant lifting" might be in this setting. Indeed, let 
\[ A = \begin{pmatrix} 1/2 & 1/2 \\ 0 & 0 \end{pmatrix} \,\, , \,\, B = \begin{pmatrix} 1/\sqrt{2} & 1/\sqrt{2} \\ 0 & 0 \end{pmatrix} .\]
These are clearly two commuting contractions. But (as one may tediously check) there is no contractive matrix of the form
\[
\begin{pmatrix} B & * \\ * & *  \end{pmatrix}
\]
that commutes with the $1$-dilation for $A$ given by
\[
\begin{pmatrix} A & D_{A^*} \\ D_A & -A  \end{pmatrix} .
\]

\section{Proof of Theorem \ref{thm:reg_dilations}}\label{sec:reg}

\subsection{Existence of regular $N$-dilations implies existence of regular dilations}

Suppose that for all $N$, the $k$-tuple $T_1, \ldots, T_k$ has a regular unitary $N$-dilation. It follows the function $T(\cdot) : \mb{Z}^k \rightarrow B(H)$ given by 
\bes
T(n_1, \ldots, n_k) = (T_1^{(n_1)_-} \cdots T_k^{(n_k)_-})^* T_1^{(n_1)_+} \cdots T_k^{(n_k)_+}
\ees
is positive definite (here $n_+ := \max\{n,0\}$ and $n_- := \max\{-n,0\}$). By the results of Sections I.7 and I.9 in \cite{SzNF70}, $T_1, \ldots, T_k$ has a regular unitary dilation.

\subsection{Existence of regular dilations implies existence of regular $N$-dilations}
The proof is similar to the proof of Theorem \ref{thm:dilations}. We begin with an analogue of Proposition \ref{prop:point_eval}.

\begin{proposition}\label{prop:point_eval_regular}
Suppose that $T_1, \ldots, T_k$ are commuting contractions on $H$, $\dim H = n$, that 
have a regular unitary dilation $U_1, \ldots, U_k$. Fix $N \in \mb{N}$. Then there is an integer $M$, there are $M$ points $w_i, \ldots, w_M \in \mb{T}^k$ and there are $M$ non-negative operators $A_1, \ldots, A_M \in B(H)$ with $\sum A_i = I_H$, such that
\be\label{eq:point_eval_regular}
q(T_1, \ldots, T_k) = \sum_{i=1}^M q(w_i) A_i \,\,, \,\, \textrm{ for all } q \in \cQ_N(k).
\ee
The integer $M$ can be taken to be $n(n+1) \times \dim \cQ_N(k) + 1$.
\end{proposition}

\begin{proof}
The proof is similar to the proof of Proposition \ref{prop:point_eval}. One just needs to notice that the functions in $\cQ(k)$ are all $n$-harmonic; thus one may use the Poisson kernel to get equation (\ref{eq:Poisson}) for all $q \in \cQ_N(k)$. The rest of the proof is the same, except for the dimension count.
\end{proof}

As in the previous section, the above proposition immediately implies a von Neumann type inequality for functions in $\cQ_N(k)$. 

The proof of Theorem \ref{thm:reg_dilations} is completed in precisely the same way as the proof of Theorem \ref{thm:dilations}, by dilating the $A_i$'s to a family of pairwise orthogonal projections.

\begin{remark}
In \cite[Section 4]{LevyShalit10}, it was shown that if $T_1, \ldots, T_k$ is a {\em doubly commuting} tuple of contractions (meaning that operators commute and also $T_i T_j^* = T_j^* T_i$ for all $i \neq j$), then it has a regular unitary $N$-dilation, and an equation such as (\ref{eq:point_eval_regular}) was obtained as a corollary. All doubly commuting $k$-tuples of contractions satisfy (\ref{eq:oi}), so one may also apply Theorem \ref{thm:reg_dilations} to see that a doubly commuting $k$-tuple has an $N$-regular dilation. However, the proof from \cite{LevyShalit10} not only provides a much smaller dimension on which the dilation acts (it is only $n(N+1)^k$) and a smaller number of points are needed for (\ref{eq:point_eval_regular}), but --- more importantly --- it provides an algorithm for constructing the dilation as well as for finding the points $w_i$ and the operators $A_i$.
\end{remark}

\newpage

\section{Acknowledgments}

The first author was partially supported by National Science Foundation Grant DMS 0966845.
The second author would like to thank Ken Davidson for the warm and generous hospitality provided at the University of Waterloo.

\bibliographystyle{amsplain}

\end{document}